\documentclass[12pt, a4paper, reqno, english]{amsart}
\usepackage[all]{xy}
\usepackage{amssymb,amsmath,amsthm}
\numberwithin{equation}{section}
\usepackage[margin=1in]{geometry}
\usepackage{comment}

\usepackage[colorlinks=true, linkcolor=blue, citecolor=green, urlcolor=black]{hyperref}

\newcommand{\N}{\mathbb{N}}
\newcommand{\Z}{\mathbb{Z}}

\newtheorem{thm}{Theorem}[section]
\newtheorem{prop}{Proposition}[section]
\newtheorem{lem}{Lemma}[section]
\newtheorem{rem}{Remark}[section]
\newtheorem{cor}{Corollary}[section]

\theoremstyle{definition}

\renewcommand{\mod}[1]{\hspace{-2.9mm}\pmod{#1}}

\newcommand{\ben}{\begin{enumerate}}
\newcommand{\een}{\end{enumerate}}
\newcommand{\eit}{\begin{itemize}}
\newcommand{\beq}{\begin{equation}}
\newcommand{\eeq}{\end{equation}}

\newcommand{\cal}{\mathcal}

\usepackage{color}
\definecolor{red}{rgb}{1,0,0}
\definecolor{blue}{rgb}{.2,.6,.75}
\definecolor{green}{rgb}{.4,.7,.4}

\renewcommand{\leq}{\leqslant}
\renewcommand{\geq}{\geqslant}
\renewcommand{\d}{\mathrm{d}}

\begin{document}

\title{Moments OF THE NUMBER OF REPRESENTATIONS AS SUMS OF TWO PRIME SQUARES}

\author{Haozhe Gou}
\date{}

\address{School of Mathematics, Shandong University, Jinan 250100, China}
\address{D\'epartement de math\'ematiques et de statistique, Universit\'e de Montr\'eal, C.P.~6128, succ.~Centre-ville, Montr\'eal, QC H3C~3J7, Canada}
\email{haozhegou@gmail.com}


\thanks{}

\begin{abstract}
We prove, for every fixed integer \(k\ge 4\), the correct order of magnitude for the $k$th moments of the function that counts the number of representations of an integer as sums of two prime squares. The upper bound for $k=4$ was previously known up to $\log\log\log x$, and the lower bound for $k\ge 4$ was only known conditionally on a conjectural uniform version of the Green-Tao theorem on linear equations in primes  by the work of Sabuncu \cite{Sabuncu2024}. As an application of our method, we give a simpler proof of the lower bounds for the  moments of the shifted prime divisor function, thereby recovering the lower-bound part of Gabdullin's recent result on a conjecture of Fan and Pomerance. 
\end{abstract}

\maketitle

\newcommand{\cbar}{\overline{\chi}}
\newcommand{\pbar}{\overline{\psi}}
\newcommand{\sumstar}{\sideset\and ^* \to \sum}

\section{Introduction}
Moments of arithmetic representation functions are classical objects in analytic number theory. Generally, such moments contain information about the distribution of the corresponding representation numbers, and are closely related to additive problems involving primes, sieve methods, among others. 
One particularly interesting family is given by representations of an integer as a sum of two squares, with primality restrictions on one or both of the coordinates.

More precisely, let \(\mathcal P\) denote the set of primes, we write
\[
r_0(n)=\#\{(a,b)\in\Z^2:n=a^2+b^2,a\ge 0,b>0\},
\]
\[
r_1(n)=\#\{(a,p)\in\Z\times\mathcal{P}:n=a^2+p^2,a>0\},
\]
and
\[
r_2(n)=\#\{(p,q)\in\mathcal P^2:n=p^2+q^2\}.
\]
The function \(r_0(n)\) is well understood through the classical theory of sums
of two squares. It is well-known that
\[
r_0(n)=\sum_{d\mid n}\chi_4(d),
\]
where \(\chi_4\) is the non-principal character modulo \(4\). Hence \(r_0\) is multiplicative.  

In a broader context, Blomer and Granville \cite{BlomerGranville2006} established uniform asymptotic formulae for the moments of representation numbers of general binary quadratic forms.
When specializing to the form \(x^2+y^2\), we have, for every fixed integer \(k\geq 1\),
\begin{equation}\label{eq:moment r0}
    \sum_{n\leq x} r_0(n)^k
\sim
a_k x(\log x)^{2^{k-1}-1}.
\end{equation}
The leading constant may be written explicitly as
\[
a_k
=\frac{L(1,\chi_4)^{2^{k-1}}}{\Gamma(2^{k-1})}
\left(\frac12\right)^{2^{k-1}-1}
\prod_{p\equiv 1\mod 4}
\left[
\left(1-\frac1p\right)^{2^k}
\sum_{\nu=0}^{\infty}\frac{(\nu+1)^k}{p^\nu}
\right]
\prod_{p\equiv 3\mod 4}
\left(1-\frac1{p^2}\right)^{2^{k-1}-1}.
\]

The behavior of \(r_1(n)\) and \(r_2(n)\) is more subtle than that of
\(r_0(n)\), since neither is multiplicative.  A first basic result follows from the prime number theorem
\begin{align}\label{eqn:first moment}
   \sum_{n\le x} r_1(n)
   =
   \frac{\pi}{2}\frac{x}{\log x}
   +O\!\left(\frac{x}{(\log x)^2}\right),
   \quad
   \sum_{n\le x} r_2(n)
   =
   \pi\frac{x}{(\log x)^2}
   +O\!\left(\frac{x}{(\log x)^3}\right).
\end{align}
The second moment of \(r_1(n)\) was first bounded by Rieger \cite{Rieger1968} and was later
evaluated asymptotically by Daniel \cite{Daniel2001}, who proved
\[
   \sum_{n\le x} r_1(n)^2
   =
   \left(\frac{\pi}{2}+\frac94\right)\frac{x}{\log x}
   +O\!\left(\frac{x(\log\log x)^2}{(\log x)^2}\right).
\]
Here the term \(\pi/2\) comes from the diagonal solutions
\(a^2+p^2=b^2+q^2\), while the term \(9/4\) is genuinely off-diagonal.
More recent work of Granville, Sabuncu and Sedunova \cite{GranvilleSedunovaSabuncu2024} studied the finer mass
distribution of \(r_1\). In particular, they proved that
\[
\#\{n\le x:r_1(n)= 1\}
=
\frac{\pi}{2}\frac{x}{\log x}
-
\frac{x(\log\log x)^{O(1)}}{(\log x)^{1+\delta}},
\]
where
\(\delta=1-\frac{1+\log\log 2}{\log 2}
=0.0860713320\cdots\)
is the multiplication table constant.
They further predicted a precise form of this secondary term involving a
non-constant \(1\)-periodic function \(\psi_1^*\). Remarkably, this is exactly
the same periodic function, up to constant, as the function $g$ appearing in
the recent breakthrough work of Green and Sawhney \cite{GreenSawhney2026} on the problem of estimating the limiting probability that a random permutation fixes a \(k\)-set.


For \(r_2(n)\), the low moments exhibit the prime paucity phenomenon.  Erd\"os \cite{Erdos1938}
proved that the off-diagonal solutions to
\(p_1^2+q_1^2=p_2^2+q_2^2\) are negligible in the second moment, obtaining an asymptotic formula with an error term \(O(x(\log\log x)^5/(\log x)^3)\). Rieger~\cite{Rieger1968} sharpened Erd\"os's estimate for the off-diagonal contribution, which, together with the diagonal contribution, implies that
\[
   \sum_{n\le x}r_2(n)^2
   =
   2\pi\frac{x}{(\log x)^2}
   +O\!\left(\frac{x}{(\log x)^3}\right).
\]
Blomer and Br\"udern~\cite{BlomerBrudern2008} proved an analogous result for the third moment,
\[
\begin{aligned}
       \sum_{n\le x}r_2(n)^3
   =
   4\pi\frac{x}{(\log x)^2}
   +O\!\left(\frac{x(\log\log x)^6}{(\log x)^3}\right),
\end{aligned}
\]
again showing that the diagonal solutions give the main term. They also remarked
that, by incorporating ideas of Rieger, the logarithmic factors in the error term
should be removable, leading to an error term \(O(x/(\log x)^3)\), although they
did not work out the details. Recently, Sabuncu~\cite{Sabuncu2024} gave a  refinement in this direction, obtaining \begin{align}\label{eqn:r2 third moment}   \sum_{n\le x}r_2(n)^3    =    4\pi\frac{x}{(\log x)^2}+O\!\left(\frac{x(\log\log x)^2}{(\log x)^3}\right).\end{align}
We also mention that mixed moments involving \(r_0(n)\), \(r_1(n)\), and \(r_2(n)\)
have been studied by Sedunova \cite{Sedunova2022}.


In this paper, we mainly concern ourselves with the higher moments of $r_2(n)$. This was first studied by  Sabuncu \cite{Sabuncu2024}, who obtained the expected unconditional upper bounds for all $k\ge5$ by refining previous arguments and applying Selberg’s sieve. Conditional on a plausible, conjectural uniform version of the Green--Tao theorem on linear equations in primes (Conjecture 1.5 in \cite{Sabuncu2024}), matching lower bounds are obtained for all $k\ge 4$. As for the upper bound for $k=4$, Sabuncu 
gave 
$$
\sum_{n\le x}r_2(n)^4\ll \frac{x}{\log x}\log\log\log x.
$$
Our main result below removes the  factor  $\log\log\log x$ from the upper bound for $k=4$ and establishes all lower bounds unconditionally.
\begin{thm}\label{thm:main} For every integer $k\ge 4$, we have
    \[
\sum_{n\le x}r_2(n)^k\asymp_k x(\log x)^{2^{k-1}-2k-1}.
\]
\end{thm}

Concerning the upper bounds, we in fact prove a stronger estimate for the factorial moments for all integers \(k\ge 1\); see Proposition~\ref{prop:factorial-upper}. 
As a consequence, we obtain the following refinement of the third moment, which improves   Sabuncu's result  \eqref{eqn:r2 third moment} and reaches the optimal error term anticipated by Blomer and Br\"udern~\cite{BlomerBrudern2008}.
\begin{cor}\label{cor:third moment}
We have
    \begin{equation*}
   \sum_{n\le x} r_2(n)^3
   =
   4\pi \frac{x}{(\log x)^2}
   +
   O\!\left(\frac{x}{(\log x)^3}\right).
\end{equation*}
\end{cor}
We briefly indicate the main new ideas for proving Theorem~\ref{thm:main}.
For the lower bound, we do not follow Sabuncu's lower-bound argument, nor do we
attempt to prove the corresponding conjectural uniform version of the Green--Tao theorem.  Instead, we take a different route: we replace the difficult pure
moment of \(r_2(n)\) by a mixed moment in which one prime-square representation
is kept and the remaining multiplicity is supplied by the classical representation function \(r_0(n)\).  
This turns the lower bound into a problem about congruence conditions for a single pair of primes, accessible unconditionally by primes in arithmetic progressions on average.

For the upper bound, we follow Sabuncu's Gaussian factorisation argument and Selberg
sieve framework, including the decomposition according to the largest prime factor.  The new ingredient is the use of a friable
Selberg--Delange estimate of Tenenbaum and Wu~\cite{TenenbaumWu2003} in the medium prime range,
which preserves a Dickman-type decay factor and removes the harmonic loss
responsible for the extra \(\log\log\log x\) in the fourth moment.

We also point out that our method for proving the lower bounds in Theorem \ref{thm:main} is not limited to \(r_2(n)\); it is likely applicable to a wider class of functions. 
We illustrate this with two further examples.

One natural example is the function \(r_1(n)\) introduced earlier. Combined with the lower-bound strategy here and the sieve framework of Sabuncu~\cite{Sabuncu2024}, with some minor modifications, we should obtain, for every integer $k\ge3$,
\begin{align}
    \sum_{n\le x} r_1(n)^k
    \asymp_k
    x(\log x)^{2^{k-1}-k-1}.
\end{align}

A second example is the shifted-prime divisor function
\[
    \omega^*(n)
    :=
    \#\{p\in\mathcal P : p-1\mid n\}.
\]
Gabdullin~\cite{Gabdullin2025} recently proved the conjecture of Fan 
and Pomerance \cite{FanPomerance2024} that, for every fixed integer \(k\ge 2\),
\begin{align}\label{eqn:omega*}
    \sum_{n\le x}\omega^*(n)^k
    \asymp_k
    x(\log x)^{2^k-k-1}.\end{align}
His proof of the lower bound analyzes prime tuples through the least common multiple \([p_1-1,\ldots,p_k-1]\), using a combinatorial identity for the least common multiple in terms of gcds, together with lower bounds for primes in arithmetic progressions and the treatment of possible exceptional zeros. As an application of our method, we will give a simpler proof and hence recover Gabdullin's result for the lower bound in \eqref{eqn:omega*}. 

The rest of this paper is organized as follows.  
In Section~\ref{sec:strategy} we outline the general strategy for both the 
lower and the upper bounds for the moments of \(r_2(n)\), and carry out 
an initial reduction.  
Section~\ref{sec:lem-for-lower} collects several lemmas needed for the 
lower bound.  
In Section~\ref{sec:pf-main-lower} we prove \eqref{eq:Mixed low bound} 
and then deduce the desired lower bound in Theorem~\ref{thm:main}.  
The auxiliary results and lemmas required for the upper bound are 
gathered in Sections~\ref{sec:friable-SD} and ~\ref{sec:upper-aux}.  
The proof of the upper bound in Theorem~\ref{thm:main} is given in 
Section~\ref{sec:proof-main-upper}.  
Finally, in Section~\ref{sec:omega-star-lower} we briefly explain the 
alternative proof of the lower bound for \(\omega^*(n)\).

\subsection*{Acknowledgment.}
The author would like to express his sincere gratitude to his supervisors, Professors Andrew Granville and Jianya Liu, for their constant support, encouragement and many valuable discussions and comments. The author would also like to thank the Centre de recherches math\'ematiques (CRM) for organizing the thematic program \emph{Universal Statistics in Number Theory}, which provided an excellent environment for learning about recent developments and for mathematical exchange. Some of the ideas in this work were inspired during this period. The author is also grateful to Cihan Sabuncu and Max Wenqiang Xu for helpful discussions and for carefully reading an earlier version of the manuscript and providing many detailed comments and suggestions, which improved the exposition of the paper. Finally, he gratefully acknowledges financial support from CSC.

\section{Proof Strategy and First Reductions}
\label{sec:strategy}
The moments of \(r_2(n)\) require an understanding of solutions to systems
\[
p_1^2+q_1^2=\cdots=p_k^2+q_k^2.
\]
To avoid the exceptional local behaviour caused by the even prime and same prime pairs, we let \(\cal P_{>2}\) denote the set of odd primes, and define
\[
    R_2(n) := \#\{(p, q)\in \cal P_{>2}^2:\ p < q,\ p^2+q^2 = n\}.
\]
This differs slightly from the function denoted by \(R_2(n)\) in
\cite{Sabuncu2024}, where the even prime 2 is allowed. However, this does not affect any of the central arguments or final conclusions. With our convention,
\begin{equation*}
    r_2(n) = 2R_2(n) + \mathbf{1}_{n\in 2\cal P^2} + 2\cdot \mathbf{1}_{n-4 \in \cal P_{>2}^2}.
    \label{eq:r2-R2-relation}
\end{equation*}
The indicator terms contribute only \(O_k(\sqrt x/\log x)\) to the $k$th moment, which is negligible.  Thus we have
\begin{equation}\label{eqn:R2(n)}
    \sum_{n\le x}r_2(n)^k=2^k\sum_{n\le x}R_2(n)^k+O_k\left(x^{1/2+\varepsilon}\right)
\end{equation}

\subsection{The Lower Bound.}
The direct approach to the lower bound would be to count many simultaneous
solutions of \(    p_1^2+q_1^2=\cdots=p_k^2+q_k^2
\) with all \(p_i,q_i\) prime.  After usual Gaussian factorisation this leads to a
system of affine-linear forms required to be prime simultaneously, and
Sabuncu's lower-bound argument invokes a generalized Green--Tao type theorem
to handle such systems.  

Our key point is to avoid estimating the pure moment of $R_2(n)$ directly: we keep only one
prime-square representation and use powers of \(r_0(n)\) to supply the
remaining multiplicity. That is, for a fixed integer $k\ge 2$, we consider the mixed moment
\begin{equation*}
    \mathcal{M}_k(x):=\sum_{n\le x} R_2(n)r_0(n)^{k-1}.
\end{equation*}
By H\"older's inequality, we have
\[
    \mathcal{M}_k(x)
    \le
    \left(\sum_{n\le x}R_2(n)^k\right)^{1/k}
    \left(\sum_{n\le x}r_0(n)^k\right)^{(k-1)/k}.
\]
Hence
\[
    \sum_{n\le x}R_2(n)^k
    \ge
    \frac{\mathcal{M}_k(x)^k}
    {\left(\sum_{n\le x}r_0(n)^k\right)^{k-1}}.
\]
Using the known moment estimate of $r_0(n)$
from \eqref{eq:moment r0}, we see that it is enough to prove
\begin{equation}\label{eq:Mixed low bound}
    \mathcal{M}_k(x)\gg_k x(\log x)^{2^{k-1}-3}.
\end{equation}
Indeed, \eqref{eq:Mixed low bound} would imply
\[
    \sum_{n\le x}R_2(n)^k
    \gg_k
    \frac{x^k(\log x)^{k(2^{k-1}-3)}}
    {\left(x(\log x)^{2^{k-1}-1}\right)^{k-1}}
    =
    x(\log x)^{2^{k-1}-2k-1},
\]
which is precisely the desired lower bound in \eqref{eqn:R2(n)}.

We now only need to consider one  representation $n=p^2+q^2$, while the remaining factor $r_0(n)^{k-1}$ can be exploited through the divisor identity
\[
    r_0(n)=\sum_{d\mid n}\chi_4(d).
\]
Squarefree divisors composed of primes $\ell\equiv1\pmod4$ give a positive contribution. Thus it remains to count prime pairs $(p,q)$ satisfying congruences of the form
\[
    m\mid p^2+q^2.
\]
After restricting $m$ to a suitable range below the Bombieri--Vinogradov level, this count is accessible unconditionally by primes in arithmetic progressions on average. The resulting Euler product gives the factor $(\log x)^{2^{k-1}-1}$, and the two prime variables contribute the factor $(\log x)^{-2}$, yielding the desired lower bound for $\mathcal{M}_k(x)$.


\subsection{The Upper Bound}
Our proof of the upper bound follows the general framework developed by
Sabuncu \cite{Sabuncu2024}, with one additional input in the medium
prime range.  
As in \cite{Sabuncu2024}, we begin by writing the \(k\)-th moment as a linear combination of factorial moments via the Stirling numbers of the second kind
\begin{equation}\label{eqn:R2-Sj-relation}
       \sum_{n\le x}R_2(n)^k
       =
       \sum_{j=1}^k
       \left\{\begin{matrix}k\\ j\end{matrix}\right\}
       \mathcal{S}_j(x),
\end{equation}
where
\begin{equation}
    \mathcal{S}_j(x):=\sum_{n\le x}R_2(n)(R_2(n)-1)\cdots(R_2(n)-j+1).
\end{equation}

Therefore the upper bound for the \(k\)-th moment of \(r_2(n)\) is controlled by the largest of \(\mathcal S_j(x)\), \(1\le j\le k\). Our goal is the following.
\begin{prop}
\label{prop:factorial-upper}
Let \(k\ge 1\) be a fixed integer. Then for all sufficiently large \(x\), we have
\[
    \mathcal{S}_k(x) \ll_k x(\log x)^{2^{k-1}-2k-1}.
\]
\end{prop}
Assuming Proposition \ref{prop:factorial-upper}, the upper bound in Theorem \ref{thm:main} follows immediately by \eqref{eqn:R2-Sj-relation} and \eqref{eqn:R2(n)}. Indeed, for \(k \ge 4\), the exponent \(2^{k-1}-2k-1\) is a strictly increasing function of \(k\). The cases \(k=2,3\) will be
used below to derive Corollary \ref{cor:third moment}.

\begin{proof}[Proof of Corollary \ref{cor:third moment} assuming Proposition \ref{prop:factorial-upper}]
It follows from \eqref{eqn:R2(n)} and \eqref{eqn:R2-Sj-relation} that
\[
   \sum_{n\le x}r_2(n)^3
   =
   8\mathcal S_1(x)+24\mathcal S_2(x)+8\mathcal S_3(x)
   +O\!\left(\frac{\sqrt x}{\log x}\right).
\]
By the first moment estimate \eqref{eqn:first moment} and  Proposition~\ref{prop:factorial-upper} with \(k=2, 3\) we obtain the result as claimed.
\end{proof}

The remainder of our upper bound argument is devoted to establishing Proposition \ref{prop:factorial-upper}. The case
\(k=1\) is just the first moment estimate, so we focus on \(k\ge 2\).  

We proceed by recalling the foundational sieve setup developed in \cite[Section~4]{Sabuncu2024}.
Let \((\vec p,\vec q)\) be an ordered \(k\)-tuple counted by
\(\mathcal S_k(x)\), and write
\[
    n=p_i^2+q_i^2\qquad (1\le i\le k).
\]
Put \(p' = P^+(n)\) and write \(n = p' N\). Since \(p_i\) and \(q_i\) are distinct odd primes, all  prime divisors of \(n\) are 2 or congruent to \(1 \pmod 4\). In particular, we may write
\[
    p'=r^2+s^2
\]
for some integers \(r,s\).
By factoring \(n\) over the Gaussian integers \(\mathbb{Z}[i]\), we may write
\[
    N=m_i^2+n_i^2,\quad
    p_i=m_ir-n_is,\quad q_i=n_ir+m_is
    \quad (1\le i\le k).
\]
Thus, once the components \((\vec m,\vec n)=(m_i,n_i)_{1\le i\le k}\) are fixed, the remaining task is to count pairs \((r,s)\) for which the quadratic form \(r^2+s^2\) and the \(2k\) linear forms \(p_i,q_i\) are simultaneously prime.

Let \(f_k(z; \vec{m}, \vec{n})\) denote the number of such pairs \((r,s)\) with \(r^2+s^2 \le z\). 
By Lemma 4.2 of Sabuncu \cite{Sabuncu2024}, an application of the
Selberg sieve to these \(2k+1\) prime conditions gives
\[
    f_k(z;\vec m,\vec n)
    \ll_k
    \mathfrak S(\vec m,\vec n)\frac{z}{(\log z)^{2k+1}},
\]
where \(\mathfrak{S}(\vec m,\vec n)\) is the corresponding singular series; see \eqref{eqn:singular-series} below for precise definition. 
By \cite[(4.10)]{Sabuncu2024}, the singular series satisfies
\begin{equation}\label{eqn:singular-series-split}
    \mathfrak S(\vec m,\vec n)
    \ll_k
    \sum_{d_1\mid N}
        \frac{\mu^2(d_1)(4k)^{\omega(d_1)}}{d_1}
    +
    \sum_{\substack{d_2\mid R(\vec m,\vec n)\\ (d_2,N)=1}}
        \frac{\mu^2(d_2)(4k)^{\omega(d_2)}}{d_2},
\end{equation}
where
\[
    R(\vec m,\vec n)
    :=
    \prod_{1\le i<j\le k}
    (m_in_j-m_jn_i)(m_im_j+n_in_j).
\]

We then split the tuples counted by \(\mathcal S_k(x)\) according to the size of
the largest prime factor \(p'=P^+(n)\).
For suitable $y$,
we distinguish the smooth range
\(P^+(n)\le y\), the medium range \(y<P^+(n)\le x^{1/3}\), and the large range
\(P^+(n)>x^{1/3}\).
The smooth and large parts of Sabuncu's argument are left unchanged. 

The only delicate part at which we modify the argument is the treatment of the medium range. In \cite{Sabuncu2024}, the medium range was decomposed \(e\)-adically according to 
\[
e^\ell<P^+(n)\le e^{\ell+1}, \quad \log y\le \ell\le \frac13\log x . 
\] 
After the largest prime factor has been extracted, the Selberg sieve reduces the contribution of this interval to an average of the singular series over tuples \((\vec m,\vec n)\) with 
\[ N\le {x}/{e^\ell}, \quad P^+(N)\le e^{\ell+1}. \] 
By \eqref{eqn:singular-series-split}, the singular series is then split into the part supported on divisors of \(N\) and the part supported on divisors of \(R(\vec m,\vec n)\). The former is handled by divisor-sum mean-value estimates, while the latter is treated by the condition-separation argument, and the main estimate used in this stage is Lemma~3.4 of \cite{Sabuncu2024}, which gives
\begin{align} \label{lem:sabuncu-lemma-3-4} 
\sum_{\substack{a^2+b^2\le x\\ d\mid a\\ P^+(a^2+b^2)\le y}} r_0(a^2+b^2)^k \ll_h \frac{x}{\log x} (\log y)^{2^k} d^{-1/2}, 
\end{align}
for large $y\le x$ and \(d\in\mathbb N\).  This leads, in effect, to a medium-range bound of the form \begin{align}\label{eqn:old-medium}
\frac{x}{\log x} \sum_{\log y\le \ell\le \log x/3} \ell^{2^{k-1}-2k-1}. 
\end{align}
This is already of the expected order for \(k\ge5\), since the exponent \(2^{k-1}-2k-1\) is then positive. For \(k=4\), the exponent is \(-1\), and one obtains an extra factor \(\log\log\log x\). For \(k=2,3\), the same treatment leaves a \((\log\log x)^2\)-loss. 

The strength of \eqref{lem:sabuncu-lemma-3-4} is that it fully exploits the
divisibility condition \(d\mid a\), giving the uniform saving \(d^{-1/2}\).
Our improvement comes from making more use of the additional smooth condition. More precisely, uniformly in the medium range, we apply the friable Selberg--Delange estimate of Tenenbaum and Wu~\cite{TenenbaumWu2003} (see Theorem \ref{lem:TW-cor23}).  
Combining this with \eqref{lem:sabuncu-lemma-3-4} allows us to retain both the Dickman function decay  and the  \(d^{-\eta}\)-saving for some positive exponent $\eta$ which is enough for our purposes. As a result, the medium range is bounded by
\[
x \sum_{\log y\le \ell\le \frac13\log x} \ell^{2^{k-1}-2k-2}\Phi_k(u_\ell), 
\] 
for some function \(\Phi_k(u)\) decays rapidly as \(u\to\infty\). This should be compared with the Sabuncu's bound \eqref{eqn:old-medium}, in which no Dickman-type factor was present.
The precise estimates and their use in the proof of the upper bound are given in Sections~\ref{sec:friable-SD}, \ref{sec:upper-aux} and \ref{sec:proof-main-upper}.

\section{Some Lemmas for Lower Bounds}\label{sec:lem-for-lower}
In this section we present three lemmas used in the proof of the lower bound. 
Throughout this section, we let \(P\) be large and write
\[
    \Pi:=\#(\mathcal P\cap I),
\]
with $I=\left[ P/4, P/3\right]$. Let \(m\le P^{1/3}\) be squarefree and assume that every prime divisor of \(m\) is congruent to \(1\pmod 4\). Define
\[
    N_m
    :=
    \#\{(p,q)\in(\mathcal P\cap I)^2:m\mid p^2+q^2\}.
\]
For \((a,m)=1\), let $\pi_I(a; m)$ denote the number of primes in $I$ congruent to $a \pmod m$,
and define
\[
    E(m):=\max_{(a,m)=1}|E_a(m)|:=
    \max_{(a,m)=1}
    \left|
        \pi_I(a;m)-\frac{\Pi}{\varphi(m)}
    \right|.
\]
\begin{lem}[Counting prime pairs with \(m\mid p^2+q^2\)]\label{lem:Nm-count} 
We have
\[
    N_m
    =
    2^{\omega(m)}\frac{\Pi^2}{\varphi(m)}
    +
    O\bigl(2^{\omega(m)}\Pi E(m)\bigr),
\]
where the implied constant is absolute.
\end{lem}

\begin{proof}
Since \(m<P/4\), every prime \(p\in I\) is coprime to \(m\). Hence for \(p,q\in\mathcal P\cap I\), the condition
\(
    m\mid p^2+q^2
\)
is equivalent to
\(
    p\equiv \nu q\pmod m
\)
for some \(\nu\pmod m\) satisfying $\nu^2\equiv-1\pmod m$, the number of such $\nu$ is $2^{\omega(m)}$. Consequently,
\[
    N_m
    =
    \sum_{\substack{\nu\bmod m\\ \nu^2\equiv -1\mod m}}
    \sum_{b\in(\mathbb Z/m\mathbb Z)^\times}
    \pi_I(b;m)\pi_I(\nu b;m).
\]
Thus, for each fixed \(\nu\),
\[
\begin{aligned}
    \sum_{b\in(\mathbb Z/m\mathbb Z)^\times}
    \pi_I(b;m)\pi_I(\nu b;m)
    &=
    \sum_{b}
    \left(\frac{\Pi}{\varphi(m)}+E_b(m)\right)
    \left(\frac{\Pi}{\varphi(m)}+E_{\nu b}(m)\right) \\
    &=
    \frac{\Pi^2}{\varphi(m)}
    +
    O(\Pi E(m)),
\end{aligned}
\]
where the cross terms vanish because
\(
    \sum_b E_b(m)=\sum_b E_{\nu b}(m)=0.
\)
Moreover,
\[
\begin{aligned}
    \left|\sum_b E_b(m)E_{\nu b}(m)\right|
    &\le
    E(m)\sum_b |E_b(m)| \le
    2\Pi E(m).
\end{aligned}
\]
Summing over \(\nu\) proves this lemma.
\end{proof}

The following lemma is a divisor-weighted version of the Bombieri--Vinogradov theorem. 
It is a standard consequence of the Bombieri--Vinogradov theorem, but we include the proof for completeness.

\begin{lem}\label{lem:weighted B-V}
For any given constants $A, B > 0$, we have
\[
\sum_{m \le P^{1/3}} \tau(m)^B \max_{(a,m)=1} \left| \pi_I(a; m) - \frac{\Pi}{\varphi(m)} \right| \ll_{A,B} \frac{P}{(\log P)^A},
\]
where $\tau(m)$ is the divisor function.
\end{lem}
\begin{proof}
Applying the Cauchy-Schwarz inequality, we obtain
\[
\sum_{m \le P^{1/3}} \tau(m)^B E(m) \le \left( \sum_{m \le P^{1/3}} \tau(m)^{2B} E(m) \right)^{1/2} \left( \sum_{m \le P^{1/3}} E(m) \right)^{1/2}.
\]
The second factor can be estimated by invoking the Bombieri-Vinogradov theorem. For any given constant $A > 0$, we have
\[
\sum_{m \le P^{1/3}} E(m)  \ll_A \frac{P}{(\log P)^A}.
\]
Next, we estimate the first factor using the trivial bound for $E(m)$.
By the lower bound for the Euler totient function, we have 
\(
E(m) \le \max_{(a,m)=1} \pi_I(a; m) + \frac{\Pi}{\varphi(m)} \ll \frac{P \log \log (m+2)}{m}.
\)
Hence
\[
\sum_{m \le P^{1/3}} \tau(m)^{2B} E(m) \ll \sum_{m \le P^{1/3}} \tau(m)^{2B}  \frac{P \log \log (m+2)}{m}\ll P (\log P)^{2^{2B}} \log \log P,
\]
where the last step follows from the classical mean value estimates $\sum_{m \le x} \frac{\tau(m)^k}{m} \ll (\log x)^{2^k}$.
Therefore the result follows by combining two estimates and choosing $A$ large enough with respect to $B$. 
\end{proof}

We also need a mean value estimate that  follows from the classical  Selberg--Delange theorem.
\begin{lem}\label{lem:euler-product}
Let \(A>0\) be fixed. Then
\[
    \sum_{\substack{m\le M\\ m\ \mathrm{squarefree} \\p\mid m\Rightarrow p\equiv1\mod4}}
    \frac{A^{\omega(m)}}{\varphi(m)}
    \asymp_A
    (\log M)^{A/2}.
\]
\end{lem}
\begin{proof}
Define the non-negative multiplicative function \(f_A(m)\) by
\[
    f_A(m):=
    \begin{cases}
    \displaystyle \frac{mA^{\omega(m)}}{\varphi(m)},
    & \text{if }\mu^2(m)=1 \text{ and } p\mid m\implies p\equiv 1\pmod 4,\\[1ex]
    0,
    & \text{otherwise}.
    \end{cases}
\]
The sum we wish to estimate is \(S(M) := \sum_{m\le M} \frac{f_A(m)}{m}\). In fact, we will give the stronger asymptotic formula for $S(M)$.

For \(\Re s > 1\), the Dirichlet series associated with \(f_A\) is given by the Euler product
\[
    F(s)
    :=
    \sum_{m=1}^{\infty}\frac{f_A(m)}{m^s}
    =
    \prod_{p\equiv 1 \mod 4}
    \left(1+\frac{Ap}{(p-1)p^s}\right).
\]
Let \(z = A/2\), and let \(\chi_0, \chi_1\) be the principal and non-principal characters modulo \(4\), respectively. We factor the \(L\)-functions as
\[
    F(s) = \zeta(s)^z G(s), \quad \text{where} \quad G(s) := (1-2^{-s})^z L(s,\chi_1)^z H(s),
\]
and \(H(s)\) is defined by the Euler product
\[
    H(s) := \prod_{p\equiv 1\mod 4} \left(1+\frac{Ap}{(p-1)p^s}\right) (1-p^{-s})^{2z} \prod_{p\equiv 3\mod 4} (1-p^{-2s})^z.
\]
For \(p\equiv 1\pmod 4\), the factor expands as
\[
    \left(1+\frac{Ap}{p-1}p^{-s}\right)(1-p^{-s})^A = 1 + \left(\frac{Ap}{p-1}-A\right)p^{-s} + O_A(p^{-2\Re s}).
\]
Since \(\frac{Ap}{p-1}-A = \frac{A}{p-1} = O_A(p^{-1})\), the linear term is \(O_A(p^{-1-\Re s})\). For \(p\equiv 3\pmod 4\), the local factor is \((1-p^{-2s})^z = 1 + O_A(p^{-2\Re s})\). Thus, the Euler product for \(H(s)\) converges absolutely in the half-plane \(\Re s > 1/2\). Because \(L(1,\chi_1) = \pi/4 \ne 0\) and \(H(1) > 0\), the function \(G(s)\) is holomorphic and non-vanishing in a neighborhood of \(s=1\).

Applying now the Selberg--Delange theorem \cite[Chapter II.5]{Tenenbaum2015} to \(F(s)\), we obtain
\[
    \sum_{m\le x} f_A(m) = \frac{G(1)}{\Gamma(z)} x(\log x)^{z-1} \left(1+o(1)\right).
\]
Finally, we transition to \(S(M)\) via partial summation
\begin{align*}
    S(M) 
    &= \frac{1}{M}\sum_{m\le M}f_A(m) + \int_1^M \frac{1}{t^2} \left(\sum_{m\le t}f_A(m)\right)\,dt. 
\end{align*}
Using \(z\Gamma(z) = \Gamma(z+1)\) and substituting \(z=A/2\), we conclude
\[
    S(M) = \frac{G(1)}{\Gamma(A/2+1)} (\log M)^{A/2} \left(1+o(1)\right) \asymp_A (\log M)^{A/2},
\]
which completes the proof.
\end{proof}

\section{Friable integers and a weighted Selberg--Delange estimate}
\label{sec:friable-SD}

For an integer \(n\geq 1\), let \(P^+(n)\) denote its largest prime factor, with the convention \(P^+(1)=1\). 
For \(\kappa>0\), let the function \(\rho_\kappa\) be defined by
\[
    \rho_\kappa(u)=0 \qquad (u\le 0),
\]
\[
    \rho_\kappa(u)=\frac{u^{\kappa-1}}{\Gamma(\kappa)}
    \qquad (0<u\le 1),
\]
and, for \(u>1\), by the differential-difference equation
\[
    u\rho_\kappa'(u)+(1-\kappa)\rho_\kappa(u)
    +\kappa\rho_\kappa(u-1)=0 .
\]
When $\kappa=1$, this is the classical Dickman--de Bruijn function $\rho=\rho_1$. It is worth noting that $\rho_\kappa$ may be viewed as the $\kappa$-th convolution power of the Dickman--de Bruijn function. 
The following standard decay estimate for \(\rho_\kappa\) will be useful.  For
each fixed \(\kappa>0\), there exists a constant \(c_\kappa\) such that
\begin{align}\label{eqn:dickman-decay}
    0\le \rho_\kappa(u)
    \ll_\kappa \exp\{-c_\kappa u\log(u+2)+O_{\kappa}(u)\}
    \qquad (u\ge 2).
\end{align}
This can be deduced from the asymptotic formula for \(\rho_\kappa(u)\) proved by Hensley~\cite{Hensley1986}, but the crude form is all that will be needed here.
In the special case \(\kappa=1\), it is consistent with the  familiar Dickman--de Bruijn decay 
\(\rho(u)=u^{-u+o(u)}\). 
For the definition and properties of the function $\rho_{\kappa}$ stated above, we also refer the reader to ~\cite{Smida1991,HildebrandTenenbaum1993}.

To bound multiplicative functions restricted to friable integers, we rely on the friable Selberg--Delange method of Tenenbaum and Wu \cite{TenenbaumWu2003}.
We state the precise form here for the convenience of the reader,.

\begin{thm}[Tenenbaum--Wu]\label{lem:TW-cor23}
Let \(\kappa>0\), \(\eta \in (0, 1/2)\), and \(\varepsilon>0\) be fixed constants. Let \(f\) be a non-negative real-valued multiplicative function. Suppose there exist positive constants \(A\) and \(C\) such that 
\[
    \sum_p\sum_{\nu\geq 2} \frac{f(p^\nu)}{p^{(1-\eta)\nu}} \leq A,
\]
and 
\[
    \sum_{p\leq z} f(p)\log p = \kappa z + O_C\left(z\exp\{-(\log z)^{3/5-\varepsilon/2}\}\right) \qquad (z>1).
\]
Then, uniformly for \(x\ge y\ge3\), and \(1\leq u := \frac{\log x}{\log y} \leq \exp\{(\log y)^{3/5-\varepsilon}\}\), we have
\[
    \sum_{\substack{n\leq x\\ P^+(n)\leq y}} f(n) 
    = C_\kappa(f)\, x(\log y)^{\kappa-1}\rho_\kappa(u) 
    \left\{ 1 + O_{A,C,\varepsilon,\eta,\kappa}\left( \frac{\log(u+1)}{\log y} + \frac{1}{(\log y)^\kappa} \right) \right\},
\]
where the constant \(C_\kappa(f)\) is given by the absolutely convergent product
\[
    C_\kappa(f) := \prod_p \left(1-\frac{1}{p}\right)^\kappa \sum_{\nu\geq 0}\frac{f(p^\nu)}{p^\nu}.
\]
\end{thm}

\begin{proof}
This is Corollaire 2.3 of Tenenbaum and Wu \cite{TenenbaumWu2003}, with their function class \(\mathcal{M}_\kappa(A,C,\eta; L_{\varepsilon/2})\) and region $H_{\varepsilon}$ written out explicitly here.
\end{proof}

Equipped with Theorem \ref{lem:TW-cor23}, we can now deduce the required bound for our specific weighted sum.

\begin{lem}[Friable Selberg--Delange bound]\label{lem:friable-SD}
Let \(k\geq 1\), \(A>0\), and \(\varepsilon\in(0, 3/5)\) be fixed. Let \(x\geq y\geq 3\). Assume that \(1\leq u=\frac{\log x}{\log y}\leq \exp\{(\log y)^{3/5-\varepsilon}\}\). Then
\[
    \sum_{\substack{n\leq x\\ P^+(n)\leq y}} r_0(n)^k \prod_{p\mid n}\left(1+\frac{A}{p}\right) 
    \ll_{k,A,\varepsilon} x(\log y)^{2^{k-1}-1}\rho_{2^{k-1}}(u).
\]
\end{lem}
\begin{proof}
Define the non-negative multiplicative function \(f(n):=r_0(n)^k \prod_{p\mid n}\left(1+\frac{A}{p}\right)\). To apply Theorem \ref{lem:TW-cor23}, we must  verify that \(f\) satisfies the required hypotheses on prime powers and primes.

First, we note that $f(p^\nu)\leq (\nu+1)^k(1+A) \ll_{k,A} (\nu+1)^k$ since \(r_0(p^\nu)=\sum_{0\le a\le \nu}\chi_4(p^a)\le \nu+1\) and 
\(1+\frac{A}{p}\leq 1+A\). Thus, choosing \(\eta=1/4\), we evaluate 
\[
    \sum_p\sum_{\nu\geq 2} \frac{f(p^\nu)}{p^{(1-\eta)\nu}} 
    \ll_{k,A} \sum_p\sum_{\nu\geq 2} \frac{(\nu+1)^k}{p^{3\nu/4}}\ll_{k,A} 1.
\]

Second, we have \(f(2)\ll 1\), \(f(p)=0\) if \(p\equiv 3\pmod 4\), and \(f(p)=2^k\left(1+\frac{A}{p}\right)\) if \(p\equiv 1\pmod 4\). Hence,
\begin{align}\label{eqn:koukou}
    \sum_{p\leq z} f(p)\log p 
    &= 2^k \sum_{\substack{p\leq z\\ p\equiv 1 \mod 4}} \log p + O_{k,A}\left(\sum_{p\leq z} \frac{\log p}{p}\right)
    \end{align}
The second sum is \(O(\log z)\). For the first sum, we use the prime number theorem in arithmetic progressions with the Vinogradov--Korobov error term, see for example
Koukoulopoulos~\cite[Theorem~1.1]{Koukoulopoulos2013}, which yields
\[
    \sum_{\substack{p\leq z\\ p\equiv 1 \mod 4}} \log p 
    = \frac{z}{2} + O\left(z\exp\{-c(\log z)^{3/5}(\log\log z)^{-1/5}\}\right).
\]
Plugging this into \eqref{eqn:koukou} we see \(f\)  matches the hypotheses of Theorem \ref{lem:TW-cor23} with parameter \(\kappa=2^{k-1}\).
Applying now Theorem~\ref{lem:TW-cor23} gives the desired bound.
\end{proof}

\section{Auxiliary estimates for the upper bound}
\label{sec:upper-aux}

Throughout this section, \(k\geq 2\) is fixed.
Define
\[
\mathcal A
:=
\left\{
(\vec m,\vec n)=(m_i,n_i)_{1\le i\le k}\in \mathbb Z^{2k}:
\begin{array}{l}
m_i^2+n_i^2=m_\ell^2+n_\ell^2\quad(1\leq i,\ell\leq k),\\
(m_i,n_i)\neq (m_\ell,n_\ell)\quad(i\neq \ell)
\end{array}
\right\}.
\]
For \((\vec m,\vec n)\in\mathcal A\), we write
\[
    N=N(\vec m,\vec n):=m_1^2+n_1^2=\cdots=m_k^2+n_k^2,
\] 
and
\[
    R(\vec m,\vec n)
    :=\prod_{1\leq i<j\leq k}
    (m_i n_j-m_j n_i)(m_i m_j+n_i n_j).\]

We shall use the following estimates from Sabuncu's work. The first is the
basic combinatorial identity for \(\mathcal A\).

\begin{lem}
\label{lem:sabuncu-A-count}
For every \(N\geq1\),
\[
    \#\left\{
    (\vec m,\vec n)\in\mathcal A:
    N(\vec m,\vec n)=N
    \right\}
    =
    k!\binom{4r_0(N)}{k}\ll_k r_0(N)^k.
\]
\end{lem}

\begin{proof}
This is \cite[(4.1)]{Sabuncu2024}.
\end{proof}

For fixed base components $(\vec m,\vec n)$, let $\nu_p(\vec m,\vec n)$ denote the number of local roots over the finite field $\mathbb{F}_p$, given by
\[
    \nu_p(\vec m,\vec n) := \#\{(r,s) \in \mathbb{F}_p^2 :  (r^2+s^2)\prod_{i=1}^k (m_ir-n_is)(n_ir+m_is) \equiv 0 \mod p\}.
\]
Define the associated singular series by the Euler product
\begin{align}\label{eqn:singular-series}
    \mathfrak{S}_k(\vec m,\vec n) := \prod_p \left(1-\frac{\nu_p(\vec m,\vec n)}{p^2}\right)\left(1-\frac{1}{p}\right)^{-(2k+1)}.
\end{align}

Let \(f_k(z;\vec m,\vec n)\) denote the number of pairs \((r,s)\) with
\(r^2+s^2\le z\) for which \(r^2+s^2\) and all the \(2k\) linear forms
\(m_ir-n_is\), \(n_ir+m_is\) are prime.
The next estimate is the Selberg sieve bound for \(f_k(z;\vec m,\vec n)\). 

\begin{lem}[Selberg sieve upper bound]
\label{lem:sabuncu-sieve}
We have, uniformly in \(z\geq 3\), 
\[
    f_k(z;\vec m,\vec n)
    \ll_k
    \mathfrak{S}_k(\vec m,\vec n)\, \frac{z}{(\log z)^{2k+1}}.
\]
\end{lem}
\begin{proof}
This is \cite[Lemma 4.2]{Sabuncu2024}.
\end{proof}

The singular series is controlled by two divisor sums.

\begin{lem}[Singular series decomposition]
\label{lem:sabuncu-singular-series}
For \((\vec m,\vec n)\in\mathcal A\), we have
\[
\mathfrak S_k(\vec m,\vec n)
\ll_k
\sum_{\ell_1\mid N}
    \frac{\mu^2(\ell_1)(4k)^{\omega(\ell_1)}}{\ell_1}
+
\sum_{\substack{\ell_2\mid R(\vec m,\vec n)\\
(\ell_2,N)=1}}
    \frac{\mu^2(\ell_2)(4k)^{\omega(\ell_2)}}{\ell_2}.
\]
\end{lem}

\begin{proof}
This is \cite[(4.10)]{Sabuncu2024}. 
\end{proof}

We also need the smooth part estimate.

\begin{lem}[Smooth part]
\label{lem:sabuncu-smooth}
Let
\(
    y=x^{1/\log\log x}.
\)
Then, for each fixed integer \(k\geq1\) and any fixed \(A>0\), we have
\[
    \sum_{\substack{n\leq x\\ P^+(n)\leq y}} r_2(n)^k
    \ll_{k,A}
    \frac{x}{(\log x)^A}.
\]
\end{lem}
\begin{proof}
Using the trivial bound \(  r_2(n)\leq r_0(n)\), 
since \(
    u:=\frac{\log x}{\log y}=\log\log x   \leq
    \exp\{(\log y)^{3/5-\varepsilon}\}
\)
for all sufficiently large \(x\), then by Lemma \ref{lem:friable-SD} we have
$$
\sum_{\substack{n\leq x\\ P^+(n)\leq y}} r_0(n)^k \ll_k x(\log y)^{2^{k-1}-1}\rho_{2^{k-1}}(u).
$$
Hence we finish the proof by \eqref{eqn:dickman-decay}, the rapid decay of the function $\rho_{\kappa}(u)$.
\end{proof}

\begin{rem}
For the smooth part, Sabuncu proves a weaker but sufficient estimate of the shape
\(    \ll_k x(\log x)^{-4}
\)
by a Rankin-trick argument in \cite[Section 5]{Sabuncu2024}. 
\end{rem}

The large-prime range is handled exactly as in Sabuncu.

\begin{lem}[Large-prime contribution]
\label{lem:sabuncu-large}
For each fixed \(k\geq2\),
\[
    \sum_{\substack{(\vec m,\vec n)\in\mathcal A\\N\le x^{2/3}}}
    f_k\!\left(\frac{x}{N};\vec m,\vec n\right)
    \ll_k
    x(\log x)^{2^{k-1}-2k-1}.
\]
\end{lem}

\begin{proof}
By Lemma \ref{lem:sabuncu-sieve}, since \(N\leq x^{2/3}\), one has
\[
\begin{aligned}
    \sum_{\substack{(\vec m,\vec n)\in\mathcal A\\N\le x^{2/3}}}
    f_k\!\left(\frac{x}{N};\vec m,\vec n\right)
    &\ll_k
    \frac{x}{(\log x)^{2k+1}}
    \sum_{\substack{(\vec m,\vec n)\in\mathcal A\\N\le x^{2/3}}}
    \frac{\mathfrak S_k(\vec m,\vec n)}{N}.
\end{aligned}
\]
It remains to estimate the singular-series sum. Splitting the singular series
by Lemma \ref{lem:sabuncu-singular-series}, the \(\ell_1\mid N\) part is
bounded by Sabuncu's argument in \cite[Section 6]{Sabuncu2024}; the
\(\ell_2\mid R(\vec m,\vec n)\) part is bounded by the condition separation
argument in \cite[Section 7]{Sabuncu2024}. Both contributions are
\[
    \ll_k(\log x)^{2^{k-1}}.
\]
This gives the lemma.
\end{proof}

We isolate the condition-separated divisor-saving estimate that will be used in the medium-prime range. For a fixed \(\ell\) in this range, put
\[
    \mathcal A_\ell
    :=
    \left\{
    (\vec m,\vec n)\in\mathcal A:
    N\leq x/e^\ell,\ P^+(N)\leq e^{\ell+1}
    \right\}.
\]
Write \(K=\binom{k}{2}\), and let \(\tau_K\) denote the \(K\)-fold divisor function.

\begin{lem}
\label{lem:sabuncu-pair-medium}
Let \(k\geq2\) be fixed, and let \(d\in\mathbb N\) squarefree. Then uniformly for
\(
    \frac{\log x}{\log\log x}\leq \ell\leq \frac13\log x
\)
we have
   $$\sum_{\substack{(\vec m,\vec n)\in\mathcal A_{\ell}\\d\mid R(\vec m,\vec n) }}1 \ll_k \frac{x}{e^\ell}\ell^{2^{k-1}-1}\frac{\tau_K(d)}{d^{1/(32K)}}.$$
\end{lem}

\begin{proof}
This is extracted from Sabuncu's medium-range treatment of \(M'_2\) in \cite[Section 8]{Sabuncu2024}. More precisely, the condition separation is the same as in \cite[(7.1)--(7.2)]{Sabuncu2024}; the mean-value input is \cite[Lemma 3.4 and (3.2)]{Sabuncu2024}. 
\end{proof}

We now average the singular series over the friable base tuples in the
medium-prime range. 
By Lemma \ref{lem:sabuncu-singular-series}, the singular series average over
\(\mathcal A_\ell\) is bounded by two contributions. We write
\[
    M_1(\ell)
    :=
    \sum_{(\vec m,\vec n)\in\mathcal A_\ell}
    \sum_{\ell_1\mid N}
    \frac{\mu^2(\ell_1)(4k)^{\omega(\ell_1)}}{\ell_1}
\]
and
\[
   M_2(\ell)
    :=
    \sum_{(\vec m,\vec n)\in\mathcal A_\ell}
    \sum_{\substack{\ell_2\mid R(\vec m,\vec n)\\(\ell_2,N)=1}}
    \frac{\mu^2(\ell_2)(4k)^{\omega(\ell_2)}}{\ell_2}
    =
    \sum_{\ell_2\geq1}
    \frac{\mu^2(\ell_2)(4k)^{\omega(\ell_2)}}{\ell_2}
    \sum_{\substack{(\vec m,\vec n)\in\mathcal A_\ell\\
                    \ell_2\mid R(\vec m,\vec n)\\
                    (\ell_2,N)=1}}
    1.
\]
Thus
\[
    \sum_{(\vec m,\vec n)\in\mathcal A_\ell}
    \mathfrak S_k(\vec m,\vec n)
    \ll_k
    M_1(\ell)+M_2(\ell).
\]

\begin{prop}[Friable average of the singular series]
\label{prop:friable-singular-average}
Let $k\geq 2$ be fixed. Write $u_{\ell} := \frac{\log(x/e^\ell)}{\ell+1}$. Then uniformly for $\frac{\log x}{\log\log x}\leq \ell\leq \frac13\log x$, we have
\begin{equation*}
\sum_{(\vec m,\vec n)\in\mathcal A_\ell} \mathfrak S_k(\vec m,\vec n) \ll_k \frac{x}{e^\ell}\ell^{2^{k-1}-1}\Phi_k(u_\ell),
\end{equation*}
where \(\Phi_k(u)\) is a positive function depending only on \(k\), which satisfies the rapid decay estimate $\Phi_k(u) \ll \exp(-c_k u \log (u+2))$ for some absolute constant $c_k > 0$.
\end{prop}
\begin{proof}
We will show $\Phi_k(u)= \rho_{2^{k-1}}(u) +  \sqrt{\rho_{2^{k-1}}(u)}$. Let's evaluate the contributions of $M_1(\ell)$ and $M_2(\ell)$ in sequence.
Using Lemma \ref{lem:sabuncu-A-count}, and applying Lemma \ref{lem:friable-SD} with \(A=4k\), \(x\) replaced by
\(x/e^\ell\), and \(y=e^{\ell+1}\), we have
\begin{equation}    \label{eq:N1-bound-short}
\begin{aligned}
    M_1(\ell)      \ll_k
    \sum_{\substack{N\leq x/e^\ell\\P^+(N)\leq e^{\ell+1}}}
    r_0(N)^k
    \prod_{p\mid N}\left(1+\frac{4k}{p}\right) \ll_k
    \frac{x}{e^\ell}\ell^{2^{k-1}-1}
    \rho_{2^{k-1}}\!\left(u_{\ell}\right).
\end{aligned}
\end{equation}

Turning to $M_2(\ell)$. For upper bounds we may drop the condition
\((\ell_2,N)=1\). By Lemma \ref{lem:sabuncu-pair-medium}, for every
\(\ell_2\geq1\),
\begin{equation}\label{eq:pair-rough-short}
     \sum_{\substack{(\vec m,\vec n)\in\mathcal A_\ell\\
                    \ell_2\mid R(\vec m,\vec n)}}
    1
    \ll_k
    \frac{x}{e^\ell}\ell^{2^{k-1}-1}
    \frac{\tau_K(\ell_2)}{\ell_2^{1/(32K)}}.
\end{equation}
On the other hand, after dropping the condition
\(\ell_2\mid R(\vec m,\vec n)\), using
Lemma \ref{lem:sabuncu-A-count}, and applying Lemma \ref{lem:friable-SD}
with \(A=0\), we get
\begin{equation}
\sum_{\substack{(\vec m,\vec n)\in\mathcal A_{\ell}}}
1  \ll_k \sum_{\substack{N\le x/e^{\ell}\\P^+(N)\le e^{\ell+1}}}r_0(N)^k                                                      \ll_k
    \frac{x}{e^\ell}\ell^{2^{k-1}-1}
    \rho_{2^{k-1}}\!\left(u_{\ell}\right).
    \label{eq:pair-friable-short}
\end{equation}
Taking the geometric mean of \eqref{eq:pair-rough-short} and
\eqref{eq:pair-friable-short}, we find 
\[
    \sum_{\substack{(\vec m,\vec n)\in\mathcal A_\ell\\
                    \ell_2\mid R(\vec m,\vec n)}}
    1
    \ll_k
    \frac{x}{e^\ell}\ell^{2^{k-1}-1}
    \rho_{2^{k-1}}(u_\ell)^{1/2}
    \frac{\tau_K(\ell_2)^{1/2}}{\ell_2^{1/(64K)}}.
\]
Therefore
\[
\begin{aligned}
M_2(\ell)&\ll_k
\frac{x}{e^\ell}\ell^{2^{k-1}-1}
\rho_{2^{k-1}}\!\left(u_{\ell}\right)^{1/2}
\sum_{\ell_2\geq 1}
\frac{\mu^2(\ell_2)(4k)^{\omega(\ell_2)}
      \tau_K(\ell_2)^{1/2}}
     {\ell_2^{1+1/(64K)}}.
\end{aligned}
\]
The last Euler product converges absolutely. Combining this with
\eqref{eq:N1-bound-short} proves the proposition.
\end{proof}

\section{Proof of Theorem \ref{thm:main}: the Upper Bound}
\label{sec:proof-main-upper}

We now prove  Proposition \ref{prop:factorial-upper} and hence Theorem \ref{thm:main} follows.
The case \(k=1\) is the first moment estimate. Hence we assume \(k\geq2\).

Let
\(    y:=x^{1/\log\log x}.
\)
We follow Sabuncu's largest-prime-factor decomposition. Splitting according to the size of \(P^+(n)\), we write
\[
    \mathcal S_k(x)
    \leq
    \mathcal S_k^{\mathrm{sm}}(x)
    +
    \mathcal S_k^{\mathrm{lg}}(x)
    +
    \mathcal S_k^{\mathrm{med}}(x),
\]
where the three pieces correspond respectively to
\[
    P^+(n)\leq y,\qquad
    P^+(n)>x^{1/3},\qquad
    y<P^+(n)\leq x^{1/3}.
\]
The smooth part is negligible by Lemma
\ref{lem:sabuncu-smooth}.
The large-prime range is treated by Lemma
\ref{lem:sabuncu-large}, which gives the correct order of upper bound. 

It remains to deal with the medium-prime range. We decompose the range of $P^+(n)$ into $e$-adic intervals of the form $(e^\ell, e^{\ell+1}]$, with $\log y \leq \ell \leq \frac{1}{3}\log x$. For a fixed interval, by Lemma \ref{lem:sabuncu-sieve}, the number of pairs $(r,s)$ is bounded by
\[
    f_k(e^{\ell+1}, (\vec{m}, \vec{n})) \ll_k \mathfrak{S}_k(\vec{m},\vec{n}) \frac{e^{\ell+1}}{(\log(e^{\ell+1}))^{2k+1}}.
\]
Summing over all $e$-adic intervals gives
\[
\begin{aligned}
     \mathcal{S}_k^{\mathrm{med}}(x) &\ll_k \sum_{\log y \leq \ell \leq \frac{1}{3}\log x} \frac{e^\ell}{\ell^{2k+1}} \sum_{\substack{(\vec{m},\vec{n}) \in \mathcal{A}_{\ell}}} \mathfrak{S}_k(\vec{m},\vec{n})\\
     & \ll_k x \sum_{\log y \leq \ell \leq \frac{1}{3}\log x} \ell^{2^{k-1}-2k-2} \Phi_k(u_{\ell}),
\end{aligned}
\]
where we used Proposition \ref{prop:friable-singular-average} in the second line. Note that \(\Phi_k(u)\ll_k e^{-c_k u\log u}\ll_{k,B}u^{-B}\) for \(u\geq2\) and any fixed $B>0$, and
$$u_\ell = \frac{\log(x/e^\ell)}{\log(e^{\ell+1})} =\frac{\log x}{\ell} - 1 + O\left(\frac{\log x}{\ell^2} \right)\asymp \frac{\log x}{\ell}$$
uniformly for $\ell\leq \frac{1}{3}\log x$.
Then by taking $B$ large enough, we have
\begin{align*}
    \mathcal S_k^{\mathrm{med}}(x) &\ll_k x(\log x)^{-B}\sum_{\ell\leq \frac{1}{3}\log x}\ell^{2^{k-1}-2k-2+B}
    \ll_k x(\log x)^{2^{k-1}-2k-1}.
\end{align*}
Combining the estimates for the smooth, medium, and large ranges completes the proof.

\section{Proof of Theorem \ref{thm:main}:  the Lower Bound}\label{sec:pf-main-lower}
We now establish the lower bound \eqref{eq:Mixed low bound}, which we restate as follows:
\[
\mathcal{M}_k(x):=\sum_{n\le x} R_2(n)r_0(n)^{k-1}\gg_k x(\log x)^{2^{k-1}-3}.
\]
Let $P = \sqrt{x}$, $I = [P/4, P/3]$, and $M = P^{1/3}$. To ensure $n \le x$, we restrict our attention to integers of the form $n = p^2 + q^2$ with $p, q \in \mathcal{P} \cap I$, since $p^2+q^2 \le 2(P/3)^2 < P^2 = x$. This gives the initial lower bound
\begin{equation}\label{eqn:initial-lower-bound-of-M}
        \mathcal{M}_k(x) \ge \sum_{p,q\in\mathcal P\cap I} r_0(p^2+q^2)^{k-1}.
\end{equation}

Recall that \(r_0(n) = \sum_{d \mid n} \chi_4(d)\) is a non-negative multiplicative function. For \(n = p^2+q^2\) where \(p\) and \(q\) are primes, we have
\begin{align*}
    r_0(p^2+q^2)^s &= \Bigg(\prod_{\ell^a \parallel p^2+q^2} \sum_{j=0}^{a} \chi_4(\ell^j)\Bigg)^s \ge \Bigg(\prod_{\substack{\ell \mid p^2+q^2 \\ \ell \equiv 1 \mod 4}} 2\Bigg)^s 
    =  \prod_{\substack{\ell \mid p^2+q^2 \\ \ell \equiv 1 \mod 4}} \bigl(1 + (2^s-1)\bigr)
    \\
    &= \sum_{\substack{m \mid p^2+q^2 \\ m\ \mathrm{squarefree} \\ \ell \mid m \implies \ell \equiv 1 \mod 4}} (2^s-1)^{\omega(m)}
    \ge
    \sum_{\substack{m\mid p^2+q^2\\ m\le M,\ m\ \mathrm{squarefree}\\\ell\mid m\implies \ell\equiv1\mod4 }}
    (2^s-1)^{\omega(m)}
\end{align*}
Substituting this with $s=k-1$ into \eqref{eqn:initial-lower-bound-of-M} and interchanging the order of summation yields
\begin{align}\label{eqn:Nm-occurs}
    \mathcal{M}_k(x)
    \ge \sum_{\substack{m \le M, \, \mu^2(m)=1 \\ \ell \mid m \implies \ell \equiv 1 \mod 4}} (2^{k-1}-1)^{\omega(m)} N_m,
\end{align}
where $ N_m
    :=
    \#\{(p,q)\in(\mathcal P\cap I)^2:m\mid p^2+q^2\}.$
By Lemma \ref{lem:Nm-count}, the inner count is given by $N_m = 2^{\omega(m)}\frac{\Pi^2}{\varphi(m)} + O\bigl(2^{\omega(m)}\Pi E(m)\bigr)$. Inserting this into \eqref{eqn:Nm-occurs} allows us to separate the main term and the error term
\[
    \mathcal{M}_k(x) \ge \Pi^2 \sum_{\substack{m \le M, \, \mu^2(m)=1 \\ \ell \mid m \implies \ell \equiv 1 \mod 4}} \frac{(2^k-2)^{\omega(m)}}{\varphi(m)} + O\Bigg( \Pi \sum_{\substack{m \le M, \, \mu^2(m)=1 \\ \ell \mid m \implies \ell \equiv 1 \mod 4}} (2^k-2)^{\omega(m)}E(m) \Bigg).
\]

We evaluate the main term first. Applying Lemma \ref{lem:euler-product} with the constant $C = 2^k-2$, and noting that $\log M \asymp \log P \asymp \log x$, we deduce
\[
    \sum_{\substack{m \le M, \, \mu^2(m)=1 \\ \ell \mid m \implies \ell \equiv 1 \mod 4}} \frac{(2^k-2)^{\omega(m)}}{\varphi(m)} \asymp_k (\log M)^{(2^k-2)/2} \asymp_k (\log x)^{2^{k-1}-1}.
\]
By the Prime Number Theorem, $\Pi = \#(\mathcal{P} \cap I) \asymp P/\log P$. Consequently, the main term is bounded below by
\[
    \gg_k \left(\frac{P}{\log P}\right)^2 (\log x)^{2^{k-1}-1} \asymp_k x (\log x)^{2^{k-1}-3}.
\]

It remains to bound the error term. We can choose a sufficiently large constant $B = B_k > 0$ such that $(2^k-2)^{\omega(m)} \le \tau(m)^B$ for all $m$. By Lemma \ref{lem:weighted B-V}, we may choose large $A > 3$ such that the error term becomes
\[
    \Pi \sum_{\substack{m \le M, \, \mu^2(m)=1 \\ \ell \mid m \implies \ell \equiv 1 \mod 4}} (2^k-2)^{\omega(m)}E(m)
    \ll_{A,k} \Pi \frac{P}{(\log P)^A}
    \ll_{A,k} \frac{x}{(\log x)^{A+1}}.
\]
which is negligible.
This establishes \eqref{eq:Mixed low bound} and completes the proof of the lower bound in Theorem \ref{thm:main}.

\section{The lower bound for the shifted-prime divisor function}
\label{sec:omega-star-lower}

Recall that $\omega^*(n):=\#\{p\in\mathcal P:p-1\mid n\}$. In this section we give a sketch proof that, for each fixed integer $k\geq 2$,
\begin{equation}\label{eq:omega*low}
    \sum_{n\le x}\omega^*(n)^k
    \gg_k
    x(\log x)^{2^k-k-1}.
\end{equation}

To apply our lower-bound strategy, we choose the auxiliary function $\tau(n)$ here.
By the standard divisor moment bound $\sum_{n\le x}\tau(n)^k \ll_k x(\log x)^{2^k-1}$ and Hölder's inequality again, it suffices to prove the mixed moment lower bound
\begin{equation}\label{eq:omega-star-mixed}
    \sum_{n\le x}\omega^*(n)\tau(n)^{k-1}
    \gg_k
    x(\log x)^{2^k-2}.
\end{equation}

To estimate the left-hand side, we begin with the inequality
$$\tau(n)^{s} \ge \sum_{m\mid n, \mu^2(m)=1} (2^{s}-1)^{\omega(m)}.$$
Indeed, this is immediate by noting that both sides are multiplicative. Plugging this into \eqref{eq:omega-star-mixed} and
swapping the order of summation, we obtain
\begin{equation}\label{eq:sum-swap}
\begin{aligned}
    \sum_{n\le x}\omega^*(n)\tau(n)^{k-1}
    &\ge
    \sum_{p\le x^{1/4}}
    \sum_{\substack{m\le x^{1/4}\\ \mu^2(m)=1}}
        (2^{k-1}-1)^{\omega(m)}
        \left\lfloor\frac{x}{[p-1,m]}\right\rfloor \\
    &\gg
    x \sum_{p\le x^{1/4}}\frac1{p-1}
    \sum_{\substack{m\le x^{1/4}\\ \mu^2(m)=1}}(2^{k-1}-1)^{\omega(m)}
        \frac{(p-1,m)}{m}.
\end{aligned}
\end{equation}
Now we use the identity $(p-1,m)=\sum_{c\mid(p-1,m)}\varphi(c)$. Writing $m = bc$ and restricting to squarefree $c \le x^{1/8}$, the inner sum over $m$ in \eqref{eq:sum-swap} is bounded below by
\begin{equation}\label{eq:m-sum}
    \sum_{\substack{c\mid p-1\\ \mu^2(c)=1\\ c\le x^{1/8}}}
        \frac{\varphi(c)(2^{k-1}-1)^{\omega(c)}}{c}
        \sum_{\substack{b\le x^{1/4}/c\\ \mu^2(b)=1\\ (b,c)=1}}
        \frac{(2^{k-1}-1)^{\omega(b)}}{b}.
\end{equation}

To estimate the inner sum over \(b\) in \eqref{eq:m-sum}, we apply the Selberg--Delange method here again, just like in the proof of Lemma \ref{lem:euler-product}.  Since \(c\le x^{1/8}\), we have
\(
    \log(x^{1/4}/c)\asymp \log x.
\)
Thus, uniformly for such \(c\),
\begin{equation}\label{eq:b-sum-eval}
    \sum_{\substack{b\le x^{1/4}/c\\ \mu^2(b)=1\\ (b,c)=1}}
        \frac{(2^{k-1}-1)^{\omega(b)}}{b}
    \gg_k
    (\log x)^{2^{k-1}-1}
    \prod_{\ell\mid c}
    \left(1+\frac{2^{k-1}-1}{\ell}\right)^{-1}.
\end{equation}
Substituting \eqref{eq:b-sum-eval} into \eqref{eq:m-sum}. 
We define \(g_k\) to be the non-negative multiplicative function supported on squarefree integers and
given on primes by
\begin{equation}\label{eq:gk-def}
    g_k(\ell)
    =
    \frac{(2^{k-1}-1)(1-1/\ell)}
         {1+(2^{k-1}-1)/\ell}
\end{equation}
so that 
\[    \frac{\varphi(c)}{c}     (2^{k-1}-1)^{\omega(c)} \prod_{\ell\mid c}     \left(1+\frac{2^{k-1}-1}{\ell}\right)^{-1}     =     g_k(c),\]
since \(c\) is squarefree. Then the contribution of the \(m\)-sum in \eqref{eq:sum-swap} is
\begin{equation}\label{eq:m-sum-simplified}
    \gg_k
    (\log x)^{2^{k-1}-1}
    \sum_{\substack{c\mid p-1\\ c\le x^{1/8}}} g_k(c).
\end{equation}

Inserting this back into \eqref{eq:sum-swap} and restricting the inner sum further to $c \le p^{1/3}$, we get
\begin{equation*}
    \sum_{n\le x}\omega^*(n)\tau(n)^{k-1}
    \gg_k
    x(\log x)^{2^{k-1}-1}
    \sum_{p\le x^{1/4}}
    \frac1{p-1}
    \sum_{\substack{c\mid p-1\\ c\le p^{1/3}}}
        g_k(c).
\end{equation*}
Comparing this with \eqref{eq:omega-star-mixed}, it remains to show
\begin{equation}\label{eq:final-lower}
S:=\sum_{p\le x^{1/4}}
    \frac1{p-1}
    \sum_{\substack{c\mid p-1\\ c\le p^{1/3}}}
        g_k(c)\gg_k x (\log x)^{2^{k-1}-1}
\end{equation}

Since all terms are non-negative, we may restrict to \(c\le x^{1/24}\) and
\(x^{1/8}<p\le x^{1/4}\).  Then \(c^3\le x^{1/8}<p\), so the condition
\(c\le p^{1/3}\) is automatic.  Hence
\[
\begin{aligned}
    S
    &\ge
    \sum_{\substack{c\le x^{1/24}\\ \mu^2(c)=1}}
    g_k(c)
    \sum_{\substack{x^{1/8}<p\le x^{1/4}\\ p\equiv1\mod c}}
    \frac1{p-1}.
\end{aligned}
\]
The weights \(g_k(c)\) are non-negative and divisor-bounded; say
\(    g_k(c)\ll_k \tau(c)^{B_k}\)
for some constant \(B_k\).  We decompose the prime range into dyadic intervals
\(I=(P,2P]\), with \(    x^{1/8}\le P\le x^{1/4}.\)
For such \(P\), we have \(x^{1/24}\le P^{1/3}\).  Hence, applying
Lemma~\ref{lem:weighted B-V} with \(B=B_k\) and \(a=1\), we obtain, for any
fixed \(A>0\),
\[
\begin{aligned}
    \sum_{\substack{c\le x^{1/24}\\ \mu^2(c)=1}}
    g_k(c)\pi_I(1;c)
    &=
    \Pi
    \sum_{\substack{c\le x^{1/24}\\ \mu^2(c)=1}}
    \frac{g_k(c)}{\varphi(c)}
    +
    O_{k,A}\left(\frac{P}{(\log P)^A}\right),
\end{aligned}
\]
where \(\Pi=\#\{p\in I\}\).  Since  \(p\in I=(P,2P]\) , it follows that
\[
\begin{aligned}
\sum_{\substack{c\le x^{1/24}\\ \mu^2(c)=1}}
g_k(c)
\sum_{\substack{p\in I\\ p\equiv1\mod c}}
\frac1{p-1}  &\ge
\frac1{2P}
\sum_{\substack{c\le x^{1/24}\\ \mu^2(c)=1}}
g_k(c)\pi_I(1;c) \\
&=
\frac{\Pi}{2P}
\sum_{\substack{c\le x^{1/24}\\ \mu^2(c)=1}}
\frac{g_k(c)}{\varphi(c)}
+
O_{k,A}\left(\frac1{(\log P)^A}\right).
\end{aligned}
\]
Moreover, since ${\Pi}/{P}\asymp\sum_{p\in (P,2P]} 1/p$, summing over all dyadic intervals \(I=(P,2P]\subset (x^{1/8},x^{1/4}]\) gives
\[
\begin{aligned}
S
&\gg_k
\left(\sum_{x^{1/8}<p\le x^{1/4}}\frac1p\right)
\sum_{\substack{c\le x^{1/24}\\ \mu^2(c)=1}}
\frac{g_k(c)}{\varphi(c)}
+
O_{k,A}\left((\log x)^{-A}\right)
\gg_k 
\sum_{\substack{c\le x^{1/24}\\ \mu^2(c)=1}}
\frac{g_k(c)}{\varphi(c)}.
\end{aligned}
\]
In the last step, we used Mertens' theorem and the fact that $g_k(c)/\varphi(c)$ is non-negative.

It remains to estimate the last multiplicative sum.   By the definition of \(g_k\), for $\ell$ prime,
\[
    \frac{g_k(\ell)}{\varphi(\ell)}
    =
    \frac{2^{k-1}-1}{\ell+2^{k-1}-1}
    =
    \frac{2^{k-1}-1}{\ell}+O_k\!\left(\frac1{\ell^2}\right).
\]
Using the standard Selberg--Delange estimate again we have $S\gg_k (\log x)^{2^{k-1}-1}$.
This proves \eqref{eq:final-lower}, which implies \eqref{eq:omega-star-mixed} and hence \eqref{eq:omega*low}.


\medskip


\bibliographystyle{amsalpha}
\bibliography{refs}

@article{Hensley1986,
  author  = {Hensley, Doug},
  title   = {The convolution powers of the {Dickman} function},
  journal = {Journal of the London Mathematical Society (2)},
  volume  = {33},
  number  = {3},
  year    = {1986},
  pages   = {395--406},
  doi     = {10.1112/jlms/s2-33.3.395}
}

@article{BlomerBrudern2008,
  author  = {Blomer, Valentin and Br{\"u}dern, J{\"o}rg},
  title   = {Prime paucity for sums of two squares},
  journal = {Bulletin of the London Mathematical Society},
  volume  = {40},
  number  = {3},
  year    = {2008},
  pages   = {457--462},
  doi     = {10.1112/blms/bdn026}
}

@article{BlomerGranville2006,
  author  = {Blomer, Valentin and Granville, Andrew},
  title   = {Estimates for representation numbers of quadratic forms},
  journal = {Duke Mathematical Journal},
  volume  = {135},
  number  = {2},
  year    = {2006},
  pages   = {261--302},
  doi     = {10.1215/S0012-7094-06-13522-6}
}

@article{Daniel2001,
  author  = {Daniel, Stephan},
  title   = {On the sum of a square and a square of a prime},
  journal = {Mathematical Proceedings of the Cambridge Philosophical Society},
  volume  = {131},
  number  = {1},
  year    = {2001},
  pages   = {1--22},
  doi     = {10.1017/S0305004101005163}
}

@article{Erdos1938,
  author  = {Erd{\H{o}}s, Paul},
  title   = {On additive properties of squares of primes. {I}},
  journal = {Nederl. Akad. Wetensch., Proc.},
  volume  = {41},
  year    = {1938},
  pages   = {37--41}
}

@article{FanPomerance2024,
  author = {Fan, Kai (Steve) and Pomerance, Carl},
  title  = {Shifted-prime divisors},
  year   = {2024},
  note   = {\url{https://arxiv.org/abs/2401.10427}}
}

@article{Gabdullin2025,
  author = {Gabdullin, Mikhail R.},
  title  = {Moments of the shifted prime divisor function},
  year   = {2025},
  note   = {\url{https://arxiv.org/abs/2505.24050}}
}

@article{GreenSawhney2026,
  author = {Green, Ben and Sawhney, Mehtaab},
  title  = {The proportion of permutations fixing a {$k$}-set},
  year   = {2026},
  note   = {\url{https://arxiv.org/abs/2604.28116}}
}

@article{GranvilleSedunovaSabuncu2024,
  author  = {Granville, Andrew and Sedunova, Alisa and Sabuncu, Cihan},
  title   = {The multiplication table constant and sums of two squares},
  journal = {Acta Arithmetica},
  volume  = {214},
  year    = {2024},
  pages   = {499--522},
  doi     = {10.4064/aa230828-19-4}
}

@article{HildebrandTenenbaum1993,
  author  = {Hildebrand, Adolf and Tenenbaum, G{\'e}rald},
  title   = {On a class of differential-difference equations arising in number theory},
  journal = {Journal d'Analyse Math{\'e}matique},
  volume  = {61},
  year    = {1993},
  pages   = {145--179}
}

@article{Koukoulopoulos2013,
  author  = {Koukoulopoulos, Dimitris},
  title   = {Pretentious multiplicative functions and the prime number theorem for arithmetic progressions},
  journal = {Compositio Mathematica},
  volume  = {149},
  number  = {7},
  year    = {2013},
  pages   = {1129--1149},
  doi     = {10.1112/S0010437X12000802}
}

@article{Rieger1968,
  author  = {Rieger, G. J.},
  title   = {{\"U}ber die {Summe} aus einem {Quadrat} und einem {Primzahlquadrat}},
  journal = {Journal f{\"u}r die reine und angewandte Mathematik},
  volume  = {231},
  year    = {1968},
  pages   = {89--100}
}

@article{Sabuncu2024,
  author  = {Sabuncu, Cihan},
  title   = {On the moments of the number of representations as sums of two prime squares},
  journal = {International Mathematics Research Notices},
  year    = {2024},
  number  = {11},
  pages   = {9411--9439},
  doi     = {10.1093/imrn/rnae044}
}

@article{Sedunova2022,
  author  = {Sedunova, Alisa},
  title   = {Intersections of binary quadratic forms in primes and the paucity phenomenon},
  journal = {Journal of Number Theory},
  volume  = {235},
  year    = {2022},
  pages   = {305--327},
  doi     = {10.1016/j.jnt.2021.06.035}
}

@article{Smida1991,
  author  = {Smida, Hikma},
  title   = {Sur les puissances de convolution de la fonction de {Dickman}},
  journal = {Acta Arithmetica},
  volume  = {59},
  number  = {2},
  year    = {1991},
  pages   = {123--143}
}

@book{Tenenbaum2015,
  author    = {Tenenbaum, G{\'e}rald},
  title     = {Introduction to Analytic and Probabilistic Number Theory},
  edition   = {3rd},
  series    = {Graduate Studies in Mathematics},
  volume    = {163},
  publisher = {American Mathematical Society},
  address   = {Providence, RI},
  year      = {2015}
}

@article{TenenbaumWu2003,
  author  = {Tenenbaum, G{\'e}rald and Wu, Jie},
  title   = {Moyennes de certaines fonctions multiplicatives sur les entiers friables},
  journal = {Journal f{\"u}r die reine und angewandte Mathematik},
  volume  = {564},
  year    = {2003},
  pages   = {119--166}
}

\end{document}